\documentclass[]{amsart}
\usepackage{amsmath,amsfonts, mathtools,amssymb,latexsym,amssymb,amsthm,adjustbox,mathrsfs,amsbsy, bm,tikz-cd,indentfirst, makecell,graphicx, verbatim, calc, enumerate,xy}
\usepackage[top=1in,bottom=1.3in,left=1in,right=1in]{geometry}
\usepackage[colorlinks=true,linkcolor=blue]{hyperref}
\usepackage{listings}
\usepackage{xcolor}
\theoremstyle{plain}
\newtheorem{theorem}{Theorem}[section]
\newtheorem{lemma}[theorem]{Lemma}

\newtheorem{corollary}[theorem]{Corollary}
\newtheorem{proposition}[theorem]{Proposition}
\newtheorem{remark}[theorem]{Remark}
\newtheorem{definition}[theorem]{Definition}

\newtheorem{question}[theorem]{Question}

\newtheorem{example}[theorem]{Example}
\newtheorem{point}[theorem]{}

\newcommand{\F}{\mathbb{F}}
\newcommand{\G}{\mathbb{G}}
\newcommand{\m}{\mathfrak{m}}
\newcommand{\p}{\mathfrak{p}}
\newcommand{\q}{\mathfrak{q}}
\newcommand{\n}{\mathfrak{n}}

\newcommand{\Z}{\mathbb{Z}}
\newcommand{\N}{\mathbb{N}}

\newcommand{\reg}{\operatorname{reg}}
\newcommand{\cx}{\operatorname{cx}}

\newcommand{\Ass}{\operatorname{Ass}}
\newcommand{\Syz}{\operatorname{Syz}}

\newcommand{\projdim}{\operatorname{projdim}}
\newcommand{\Hom}{\operatorname{Hom}}

\newcommand{\Ext}{\operatorname{Ext}}
\newcommand{\Tor}{\operatorname{Tor}}

\usepackage{graphicx}
\usepackage{subfigure}
\usepackage{epsfig}
\usepackage{epstopdf}
\usepackage{float}
\theoremstyle{plain}

\usepackage{xcolor}
\usepackage{titlesec}

\titleformat{name=\section}{}{\thetitle.}{0.5em}{\centering\scshape\large}

\begin{document}
\title{\large \textbf{Regularity and $v$-number of syzygies of graded maximal Cohen-Macaulay modules}}

\author{Samarendra Sahoo}
\email{samarendra.s.math@gmail.com}
\address{ Indian Institute of Technology Dharwad, Dharwad 580011, Karnataka, India}
\date{\today}


\begin{abstract}
Let $A=Q/(\mathbf f)$, where $Q$ is a polynomial ring over an infinite field $k$ and $\mathbf f=f_1,\ldots,f_c$ is a homogeneous $Q$-regular
sequence. For $1\leq j\leq c$, set $w_j=\operatorname{deg}(f_j)>0.$ Let $M$ be a finitely generated graded $A$-module of infinite projective dimension. Set $M_i=\operatorname{Syz}_i^A(M)$.  We prove that, for each $l\in\{0,1\}$, the sets $\operatorname{Ass}_A(M_{2i+l})$ eventually stabilize, and that $\operatorname{reg}(M_{2i+l})$ is eventually bounded above and below by linear polynomials. If $w_j=w$ for all $1\leq j\leq c$, then the functions
$i\longmapsto \operatorname{reg}(M_{2i+l})$, $i\longmapsto v(M_{2i+l}),$ and $i\longmapsto v_{\mathfrak p}(M_{2i+l}),\, \mathfrak p\in\mathcal A_l,$ are eventually linear with leading coefficient $w$, where $\mathcal A_l$ denotes the stable set of associated primes.
\end{abstract}

\maketitle

\section{Introduction}

Castelnuovo-Mumford regularity provides important information about graded modules, including bounds on the degrees of minimal generators of syzygies and the vanishing of local cohomology. Understanding its asymptotic behavior has been a central theme in commutative algebra. In particular, eventual linearity phenomena have been studied for powers of ideals, as well as for families of Tor and Ext modules; see, for instance,  \cite{irregularity}, \cite{trungcut}, \cite{greg},  \cite{ghoshtony}, and \cite{kodiyalam}. 
Another numerical invariant of interest is the $v$-number. The notion of the Vasconcelos invariant of a homogeneous ideal in a polynomial ring over a field was introduced in \cite{cooper} in connection with the asymptotic behavior of the minimum distance of projective Reed-Muller type codes. This invariant, now commonly referred to as the $v$-number, was later extended from homogeneous ideals to graded modules in \cite[Definition 1.4]{ghoshluca}. Its asymptotic behavior for graded modules has since been investigated in several works; see, for example, \cite{ghoshv} and \cite{ghoshs}.

Motivated by these developments, in this paper we study the asymptotic behavior of Castelnuovo-Mumford regularity and the $v$-number of syzygies of finitely generated graded modules over graded complete intersection rings. More precisely, we investigate eventual linearity, local and global $v$-numbers, and their relation with regularity.

Let $R=k[x_1,\ldots,x_n]$ be a standard graded algebra over a field $k$, and let $M$ be a finitely generated graded $R$-module. For $\mathfrak{p} \in \Ass_R(M)$, the local $v$-number of $M$ at $\p$ is defined as $v_{\p}(M)
:= \inf\{n : \p = (0 :_R x)\text{ for some } x \in M_n\}.$ The global $v$-number (or $v$-number) of $M$ is defined by $v(M)=\min\left\{v_{\p}(M):\p\in\Ass_R(M)\right\}.$ By convention, $v(0)=\infty$.


Throughout the paper, we work with the following setup.

\begin{point}\label{hypo}
\normalfont
\textbf{Hypothesis.}

Let $Q=k[X_1,\ldots,X_n]$, where $k$ is an infinite field, and let $\n=(X_1,\ldots,X_n)$ be the unique homogeneous maximal ideal
of $Q$. Let $\mathbf f=f_1,\ldots,f_c$ be a homogeneous $Q$-regular sequence contained in $\n$, and set $A=Q/(\mathbf f)$ and $\m=\n/(\mathbf f)$ be the unique homogeneous maximal ideal of $A.$ For $1\leq j\leq c$, write $w_j=\deg(f_j)>0.$ Let $\dim(A)=d=n-c$. Let $M$ and $N$ be finitely generated graded $A$-modules such that $\operatorname{projdim}_A(M)=\infty$. Let $$
\mathbb F:\quad
\cdots
\longrightarrow F_i \xrightarrow{\partial_i}
F_{i-1}
\longrightarrow\cdots\longrightarrow F_1 \xrightarrow{\partial_1}
F_0\longrightarrow M\longrightarrow 0
$$
be a minimal graded free resolution of $M$ over $A$. For all $i\geq0$, set $M_i=\Syz_i^A(M),$ with $M_0=M$.
\end{point}
 In this paper, we are mainly interested in the following questions.

\begin{question}\label{q}
Under the hypotheses of \ref{hypo} and for each $l\in \{0,1\}$,
\begin{enumerate}
    \item Does the sequence of associated-prime sets
    $\{\Ass_A(M_{2i+l})\}_{i\geq0}$ eventually stabilize?

    \item If so, for $\mathfrak p$ belonging to the stable set of associated primes, are the functions
    $i\longmapsto v_{\mathfrak p}(M_{2i+l})$ and  $i\longmapsto v(M_{2i+l})$ eventually linear?

    \item Is the function $i\longmapsto \reg(M_{2i+l})$ eventually linear?
\end{enumerate}
\end{question}

 For the definition of regularity, see \ref{defreg}. In this paper, we give affirmative answers to the above questions under suitable hypotheses. More precisely, we prove that Question \ref{q}(1) has an affirmative answer; in fact, the corresponding result holds also in the local case. If, in addition,
$\deg(f_j)=w \text{ for all }1\leq j\leq c,$
then Questions \ref{q}(2) and \ref{q}(3) also have affirmative answers. Moreover, in both cases, the eventual linear functions have leading coefficient $w$.

Since $A$ is Cohen-Macaulay, $M_i$ is maximal Cohen-Macaulay for all $i\geq\dim(A)$. As our results concern asymptotic behavior, after
replacing $M$ by a sufficiently high syzygy and reindexing, we may assume that $M$ is maximal Cohen-Macaulay. This replacement does not affect the eventual conclusions, except possibly for interchanging the even and odd subsequences.

We begin with the asymptotic behavior of the regularity of syzygies.

\begin{theorem}[Theorems \ref{linearbound} and \ref{Reg-syzygy}]
\label{first}
Assume the hypotheses of \ref{hypo}. Let $M$ be a graded maximal
Cohen--Macaulay $A$-module and let $l\in\{0,1\}$. Then the following hold.
\begin{enumerate}
    \item The function $i\longmapsto \reg(M_{2i+l})$ is eventually bounded above and below by linear polynomials.

    \item Suppose that $\deg(f_j)=w$ for all $1\leq j\leq c$. Then, for
    each $l\in\{0,1\}$, the function
    $ i\longmapsto \reg(M_{2i+l})$
    is eventually given by a linear polynomial with leading coefficient $w$.
\end{enumerate}
\end{theorem}

The preceding theorem describes the asymptotic behavior of the regularity
of the syzygies. We next turn to their associated primes and $v$-numbers.
We first show that the associated primes stabilize along each parity.

\begin{theorem}[Theorem \ref{stable-ass}]
\label{second}
Assume the hypotheses of \ref{hypo}. Let $M$ be a graded maximal Cohen-Macaulay $A$-module. Then, for each $l\in\{0,1\}$, the sets
$\Ass_A(M_{2i+l})$ are eventually stable.
\end{theorem}

In view of Theorem \ref{second}, for each $l\in\{0,1\}$ we may denote the
stable set of associated primes of $M_{2i+l}$ by $\mathcal A_l$. We then
study the asymptotic behavior of the local $v$-number at each prime in
$\mathcal A_l$. When the defining equations of the complete intersection
have the same degree, both the local and global $v$-numbers are eventually
linear, with the same leading coefficient as the regularity.

\begin{theorem}[Theorem \ref{local-v-linearity}]
\label{third}
Assume the hypotheses of \ref{hypo}, and suppose that $\deg(f_j)=w$ for all $1\leq j\leq c$. Let $M$ be a graded maximal
Cohen-Macaulay $A$-module with $\cx_A(M)\geq1$. For each
$l\in\{0,1\}$, let $\mathcal A_l=\Ass_A(M_{2i+l}) \text{for } i\gg0.$ Then, for every $\p\in\mathcal A_l$, there exists $b_{\p,l}\in\mathbb Z$ such that
$v_{\p}(M_{2i+l})=wi+b_{\p,l}$
for all $i\gg0$. Consequently,
$v(M_{2i+l})=wi+b_l$ for all $i\gg0$, where
$b_l=\min_{\p\in\mathcal A_l}b_{\p,l}.$
\end{theorem}

We also study the corresponding invariants of the free modules appearing
in the minimal graded free resolution of $M$. Under the hypotheses of
\ref{hypo}, for each $l\in\{0,1\}$, the functions
$i\longmapsto \reg(F_{2i+l}),\,
i\longmapsto v(F_{2i+l}),
\text{ and }\,
i\longmapsto v_{\p}(F_{2i+l}),$
where $\p\in\Ass_A(A)=\operatorname{Min}(A)$, are eventually linear, with
leading coefficients in $\{w_1,\ldots,w_c\}$. Moreover, we show that the eventual leading
coefficients of the local and global $v$-numbers are at most the
corresponding leading coefficient of the regularity. In the Artinian case, equality $\reg(F_i)=v(F_i)$ for all $i$ characterizes purity of the minimal graded free resolution.

\medskip

\noindent
\textbf{Organization of the paper.}
The paper is organized as follows. In Section~2, we collect the necessary preliminaries and develop the tools used throughout the paper. In
Section~3, we study the asymptotic behavior of the regularity of syzygies and prove Theorem \ref{first}. In Section~4, we establish Theorems \ref{second} and \ref{third}. Finally, in Section~5, we compare regularity and $v$-numbers of the free modules and syzygies, discuss the Artinian case, and provide examples illustrating the linearity results.

\section{Preliminaries}

Throughout this article, all rings are assumed to be commutative Noetherian rings with identity, and all modules are assumed to be finitely generated. In this section, we discuss some preliminary results that we need.

\begin{point}\label{defreg}
\normalfont
Let $A=A_0[x_1,\ldots,x_d]$ be a standard $\mathbb{N}$-graded ring, i.e.,
$\deg(A_0)=0$ and $\deg(x_i)=1$ for $1\leq i\leq d$. We say $A$ is $*$-local if $A_0$ is a local ring. Set
$A_+:=\bigoplus_{n\geq1}A_n$. Let $M$ be a finitely generated
$\mathbb{Z}$-graded $A$-module. Set
$\operatorname{end}(M)=\sup\{n\in\mathbb{Z}\mid M_n\neq0\}.$
By convention, $\operatorname{end}(0)=-\infty$. Let $H^i_{A_+}(M)$ denote the $i$th graded local cohomology module of $M$
with respect to $A_+$. Set
$a_i(M)=\operatorname{end}H^i_{A_+}(M)$ for all $i\geq0$. The
Castelnuovo-Mumford regularity of $M$ is defined to be
$$
\reg(M):=\max\{a_i(M)+i\mid i\geq0\}.
$$
Note that $M=0$ if and only if $\reg(M)=-\infty$.
\end{point}

The following well-known lemma describes the behavior of Castelnuovo-Mumford regularity with respect to short exact sequences. We will use it repeatedly throughout the paper.

\begin{lemma}\label{reg}
Let $A$ be as above. If $0\longrightarrow M'\longrightarrow M\longrightarrow M''\longrightarrow0
$ is a short exact sequence of finitely generated $\mathbb{Z}$-graded $A$-modules, then the following inequalities hold:
\begin{enumerate}
    \item[(i)] $\reg(M')\leq
    \max\{\reg(M),\reg(M'')+1\}$.

    \item[(ii)] $\reg(M)\leq
    \max\{\reg(M'),\reg(M'')\}$.

    \item[(iii)] $\reg(M'')\leq
    \max\{\reg(M')-1,\reg(M)\}$.

\end{enumerate}
 Moreover, if $M=M'\oplus M''$, then $\reg(M)=\max\{\reg(M'),\reg(M'')\}.$
\end{lemma}

\begin{point}\label{regpoly}
\normalfont
Let $Q=k[X_1,\ldots,X_d]$ be a polynomial ring over a field $k$ with the standard grading, i.e., $\deg(X_i)=1$ for $1\leq i\leq d$. Let $M$ be a finitely generated $\mathbb Z$-graded $Q$-module. In this setting, $\reg(M)$ can also be expressed in terms of the degrees of the homogeneous generators of the syzygies of $M$
(see \cite[Theorem 4.3]{Eisenbudreg}).

Let
$$
0\longrightarrow F_{d'}
\longrightarrow F_{d'-1}
\longrightarrow\cdots
\longrightarrow F_1
\longrightarrow F_0
\longrightarrow M
\longrightarrow0
$$
be a minimal graded free resolution of $M$ over $A$, where
$\displaystyle 
F_i=\bigoplus_{s\in\mathbb Z}Q(-s)^{\beta_{i,s}}$
for $0\leq i\leq d'$. The integers $\beta_{i,s}$ are called the
$(i,s)$th graded Betti numbers of $M$. It is well known that
$$
\reg(M)=\max\{s-i\mid \beta_{i,s}\neq0,\ 0\leq i\leq d'\}.
$$
Set
$\beta^Q_i(M)=\sum_{s\in\mathbb Z}\beta_{i,s},$ the $i$th total Betti number of $M$. In Hypothesis $\ref{hypo}$, since $A$ is a graded complete intersection, $\reg(A)=\sum_{j=1}^{c}(w_j-1).$
\end{point}

 We next recall the construction of Eisenbud operators in the graded setting. Throughout this section, we work under the hypotheses stated in \ref{hypo}.

\begin{point}{Eisenbud operators}
\normalfont 

Let
$$
\mathbb F:\quad
\cdots\longrightarrow F_{i+1}
\xrightarrow{\partial_{i+1}}
F_i
\xrightarrow{\partial_i}
F_{i-1}
\longrightarrow\cdots\longrightarrow
F_0\longrightarrow M\longrightarrow0
$$
be a graded free resolution of $M$ over $A$. Set $\F_i$ to be the $i$th free module in the complex $\F.$ For $1\leq j\leq c$, let $\mathbb F^{(j)}$ denote the complex obtained from $\mathbb F$ by shifting the internal degree by $-w_j$; thus $\mathbb F^{(j)}_i=F_i(-w_j)$ for all $i$. We also denote by
$\mathbb F(+2)$ the complex obtained by shifting the homological degree by two,
namely $\mathbb F(+2)_i=F_{i+2}.$

The Eisenbud operators associated to the regular sequence
$\mathbf f=f_1,\ldots,f_c$ are constructed as follows. Choose graded free
$Q$-modules $\widetilde F_i$ and homogeneous $Q$-linear maps
$\widetilde{\partial}_i:
\widetilde F_i\longrightarrow\widetilde F_{i-1}$
lifting the differentials of $\mathbb F$. Thus
$\widetilde F_i\otimes_QA\cong F_i$ and
$\widetilde{\partial}_i\otimes_QA=\partial_i$. Since $\partial^2=0$ over $A$, we have $ \widetilde{\partial}^{\,2}\equiv0
\pmod{(f_1,\ldots,f_c)}.$
Consequently, there exist homogeneous $Q$-linear maps
$\displaystyle 
\widetilde t_{j,i}:
\widetilde F_{i+2}\longrightarrow\widetilde F_i(-w_j),
\text{ for all } 1\leq j\leq c,
$
such that
$$
\widetilde{\partial}^{\,2}
=\sum_{j=1}^cf_j\widetilde t_j.
$$
Tensoring with $A$, we obtain homogeneous chain maps
$$
t'_j=t'_j(Q,\mathbf f,\mathbb F):
\mathbb F(+2)\longrightarrow\mathbb F^{(j)},
\text{ for all } 1\leq j\leq c.
$$
In homological degree $i$, this gives a homogeneous map
$t'_{j,i}:F_{i+2}\longrightarrow F_i(-w_j).$ The maps $t'_1,\ldots,t'_c$ are called the \emph{Eisenbud operators}
associated to $\mathbf f$. They are uniquely determined up to homotopy; see \cite[Corollary 1.4]{Eisenbud}. Moreover, the chain maps $t'_j$, $1\leq j\leq c$, commute up to homotopy; see \cite[Corollary 1.5]{Eisenbud}. Therefore, the maps
$$
\Hom_A(t'_j,N):
\Hom_A(\mathbb F^{(j)},N)
\longrightarrow
\Hom_A(\mathbb F(+2),N)
$$
induce well-defined homogeneous $A$-module homomorphisms
$$
t_j:\Ext_A^i(M,N)
\longrightarrow
\Ext_A^{i+2}(M,N)(-w_j)
$$
for all $i\geq0$ and $1\leq j\leq c$. Thus, $\displaystyle \Ext_A^*(M,N):=
\bigoplus_{i\geq0}\Ext_A^i(M,N)$ becomes a graded $R:=A[t_1,\ldots,t_c]$-module, where $R$ is the graded polynomial ring over $A$ in the
\emph{cohomology operators} $t_j$, with
$\deg(t_j)=2$ for $1\leq j\leq c$. These structures depend only on $\mathbf f$, are natural in both module arguments, and commute with the connecting maps induced by short exact sequences.
\end{point}

\begin{point}\label{k}
\normalfont
Since $Q$ is a polynomial ring over a field,
$\operatorname{projdim}_Q(M)<\infty$. Hence $\Ext_Q^i(M,N)=0$ for all $i\gg0$. Therefore, by \cite[Theorem 3.1]{Gulliksen},
$\Ext_A^*(M,N)$ is finitely generated over
$A[t_1,\ldots,t_c]$.

Let $\m=\n/(\mathbf f)$
be the unique homogeneous maximal ideal of $A$. Then $k=A/\m\cong Q/\n.$ Since
$\m\Ext_A^i(M,k)=0$
for all $i\geq0$, it follows that
$\Ext_A^*(M,k)$ is a finitely generated module over $S:=k[t_1,\ldots,t_c].$ Moreover,
$\dim_{\mathcal S}\Ext_A^*(M,k)=\cx_A(M).$
Recall that the complexity of $M$ is defined by
$$
\cx_A(M)
=
\inf\left\{
b\in\mathbb Z_{\geq0}\ \middle|\
\beta_n^A(M)\leq Cn^{b-1}
\text{ for some }C>0\text{ and all }n\gg0
\right\}.
$$
By \ref{hypo} and \cite[Theorem 9.2.1]{avramov1998infinite}, $\cx_A(M)\geq 1.$
\end{point}

\begin{point}{Bigraded structures}\label{str}
\normalfont

We make $R=A[t_1,\ldots,t_c]$ a $\mathbb{Z}^2$-graded ring as follows. Write
$A=A_0[x_1,\ldots,x_n]$, where $x_i$ denotes the image of $X_i$ in $A$ and $\deg(x_i)=1$ for $1\leq i\leq n$. Note that $A_0=Q_0=k$.
Thus, $R=k[x_1,\ldots,x_n,t_1,\ldots,t_c].$
Set $\deg(x_i)=(0,1)
\text{ for all }1\leq i\leq n,$ and $\deg(t_j)=(2,-w_j)
\text{ for all }1\leq j\leq c.$ We give $E^*:=\Ext_A^*(M,N)$
a $\mathbb{Z}^2$-graded structure by defining its $(i,a)$th graded component
to be $$E^*_{(i,a)}=\Ext_A^i(M,N)_a
\text{ for } (i,a)\in\mathbb{Z}^2.$$

Set $\displaystyle \Ext_A^{\mathrm{even}}(M,N) =\bigoplus_{i\in\mathbb{Z}}\Ext_A^{2i}(M,N)$
and $\displaystyle \Ext_A^{\mathrm{odd}}(M,N)=
\bigoplus_{i\in\mathbb{Z}}\Ext_A^{2i+1}(M,N).$ Following \cite[Section 2.6]{irregularity}, set
$\mathcal S:=k[X_1,\ldots,X_n,Y_1,\ldots,Y_c],$
where $\deg(X_i)=(0,1)
\text{ for all }1\leq i\leq n,$
and $\deg(Y_j)=(1,-w_j)
\text{ for all }1\leq j\leq c.$ Then $\displaystyle \Ext_A^{\mathrm{even}}(M,N)$ and
$\displaystyle \Ext_A^{\mathrm{odd}}(M,N)$ are canonically
$\mathbb{Z}^2$-graded $\mathcal S$-modules. For instance, the $(i,a)$th graded
component of $\Ext_A^{\mathrm{even}}(M,N)$ is defined by
$$
\Ext_A^{\mathrm{even}}(M,N)_{(i,a)}
=
\Ext_A^{2i}(M,N)_a
$$
for $(i,a)\in\mathbb{Z}^2$. The actions of
$X_1,\ldots,X_n,Y_1,\ldots,Y_c$ are defined by the actions of
$x_1,\ldots,x_n,t_1,\ldots,t_c$, respectively.

In particular,
$Y_j\cdot\Ext_A^{2i}(M,N)_a
\subseteq
\Ext_A^{2(i+1)}(M,N)_{a-w_j}$
for $i,a\in\mathbb{Z}$ and $1\leq j\leq c$.
By Proposition 2.7 of \cite{irregularity},
$\Ext_A^{\mathrm{even}}(M,N)$ and $\Ext_A^{\mathrm{odd}}(M,N)$ are finitely generated
$\mathbb{Z}^2$-graded modules over $\mathcal S=k[X_1,\ldots,X_n,Y_1,\ldots,Y_c].$
\end{point}

We complete this section by proving the key result of this paper, which is the bigraded analogue of Proposition 2.13 of \cite{asymptotic}. The proof is similar to the local case; however, we include the details for clarity and to keep track of the internal degree shifts arising from the bigrading.

\begin{proposition}
\label{cxM}
Let $\deg(f_j)=w$ for all
$1\leq j\leq c$. Let $M$ be a graded maximal Cohen-Macaulay $A$-module of
complexity at least one. Let $R=A[t_1,\ldots,t_c]$ be as in \ref{str}. Then there
exist $\alpha=a_1t_1+\cdots+a_ct_c\in k[t_1,\ldots,t_c]\subseteq R,
\, a_j\in k,$ of bidegree $(2,-w)$ and an integer $n_0\in\mathbb N$ such that
$\alpha$ induces homogeneous surjective maps $M_{i+2}\xrightarrow{\alpha_{i,M}}M_i(-w)$
for all $i\geq n_0$. Moreover, the following properties hold:
\begin{enumerate}
    \item[{(a)}]
    The maps
    $\Tor_j^A(\alpha_{i,M},k):
    \Tor_j^A(M_{i+2},k)
    \longrightarrow
    \Tor_j^A(M_{i},k)(-w)$
    are surjective for all $j\geq0$ and $i\geq n_0$.

    \item[{(b)}]
    Set $K_i=\ker(\alpha_{i,M})$ for all $i\geq n_0$. Then there is a short exact sequence
    $0\longrightarrow K_i\longrightarrow M_{i+2}
    \xrightarrow{\alpha_{i,M}}M_i(-w)\longrightarrow0$ with $\cx_A(K_i)=\cx_A(M)-1$ for all $i\geq n_0$.

    \item[{(c)}]  $K_{i+1}\cong\Syz_1^A(K_i)$
    for all $i\geq n_0$.
\end{enumerate}
\end{proposition}

\begin{proof}
Set $E=\Ext_A^*(M,k).$ By \ref{k}, $E$ is a finitely generated
$S=k[t_1,\ldots,t_c]$-module. Recall that
$\deg(t_j)=(2,-w)$ for all $1\leq j\leq c$. If we forget the internal grading and consider $\deg(t_i)=2$ for all $1\leq i\leq c$, then $S$ is a graded polynomial ring over the infinite field $k$.
Hence, we may choose a homogeneous linear form $\alpha=a_1t_1+\cdots+a_ct_c,
\, a_j\in k,$ such that $(0:_E\alpha)$ has finite length as an $S$-module.

Since $\deg(t_j)=(2,-w)$ for all $1\leq j\leq c$, the element
$\alpha$ is bihomogeneous of bidegree $(2,-w)$. As
$(0:_E\alpha)$ has finite length, there exists $n_0\in\mathbb N$ such that $(0:_E\alpha)_{(i,a)}=0$ for all $i\geq n_0$ and every $a\in\mathbb Z$. Consequently,
multiplication by $\alpha$ induces injective maps
\begin{equation}
\label{g}
\Ext_A^i(M,k)\xrightarrow{\alpha}\Ext_A^{i+2}(M,k)(-w)
\end{equation}
for all $i\geq n_0$. Let $\mathbb F$ be the minimal graded free resolution of $M$ over
$A$. The element $\alpha$ corresponds to the chain map
$$
\alpha'=a_1t'_1+\cdots+a_ct'_c:
\mathbb F(+2)\longrightarrow\mathbb F(-w).
$$
Denote its $i$th component by $\alpha'_i:F_{i+2}\longrightarrow F_i(-w).$
The multiplication map in \eqref{g} is induced in cohomology by $\Hom_A(\alpha',k)$.
Since $\mathbb F$ is minimal, $\Ext_A^i(M,k)\cong\Hom_A(F_i,k)$
for all $i\geq0$. Moreover, $
\Hom_A(F_i(-w),k)
\cong\Hom_A(F_i,k)(w).$
Hence \eqref{g} implies that
$$
\Hom_A(\alpha'_i,k):
\Hom_A(F_i(-w),k)
\longrightarrow
\Hom_A(F_{i+2},k)
$$
is injective for all $i\geq n_0$. Taking graded $k$-duals, we obtain
surjective maps
$$
\Hom_A(F_{i+2},k)^\vee
\longrightarrow
\Hom_A(F_i(-w),k)^\vee.
$$
Using the natural isomorphism
$\Hom_k(F_i\otimes_Ak,k)
\cong\Hom_A(F_i,k),$
we obtain surjective maps
\begin{equation}
\label{free-surj}
\alpha'_i\otimes_Ak:
F_{i+2}\otimes_Ak
\longrightarrow
F_i(-w)\otimes_Ak
\end{equation}
for all $i\geq n_0$. By the Nakayama lemma,
$\alpha'_i:F_{i+2}\longrightarrow F_i(-w)$
is surjective for all $i\geq n_0$. Since $\alpha'$ is a chain map, it induces surjective homogeneous maps $\alpha_{i,M}:M_{i+2}\longrightarrow M_i(-w)$ for all $i\geq n_0$. 
By \eqref{free-surj}, for all $i\geq n_0$ and $j\geq0$, the map
$$
\Tor_j^A(\alpha_{i,M},k):
\Tor_j^A(M_{i+2},k)
\longrightarrow
\Tor_j^A(M_i,k)(-w)
$$
is identified with the map induced by $\alpha'_{i+j}\otimes_A k,$ and is therefore surjective. This proves part (a).

For all $i\geq n_0$, set $K_i=\ker(\alpha_{i,M}).$
Then
\begin{equation}
\label{Kseq}
0\longrightarrow K_i\longrightarrow M_{i+2}
\xrightarrow{\alpha_{i,M}}M_i(-w)\longrightarrow0
\end{equation}
is exact. Since the induced maps on Tor are surjective, the long exact
sequence of Tor associated to \eqref{Kseq} yields
$$
0\longrightarrow
\Tor_j^A(K_i,k)
\longrightarrow
\Tor_j^A(M_{i+2},k)
\longrightarrow
\Tor_j^A(M_i,k)(-w)
\longrightarrow0
$$
for all $j\geq0$. Hence $\beta_j^A(K_i)=\beta_j^A(M_{i+2})-\beta_j^A(M_i)=
\beta_{j+i+2}^A(M)-\beta_{j+i}^A(M).$ Since $A$ is a complete intersection, for each $l\in \{0,1\}$, the functions $i\longmapsto \beta^A_{2i+l}(M)$ are eventually polynomial of degree $\cx_A(M)-1$ with the same leading coefficients. Therefore, 
$\cx_A(K_i)=\cx_A(M)-1$
for all $i\geq n_0$. This proves part $(b)$ (see \cite[Theorem 9.2.1]{avramov1998infinite}).

For all $i\geq n_0$, set $\overline F_i=\ker\left(
\alpha'_i:F_{i+2}\longrightarrow F_i(-w)
\right).$ Since $\alpha'_i$ is surjective and $F_i(-w)$ is free, the sequence
$$
0\longrightarrow\overline F_i
\longrightarrow F_{i+2}
\xrightarrow{\alpha'_i}
F_i(-w)
\longrightarrow0
$$
is split exact. Hence $\overline F_i$ is a graded free $A$-module. The differentials of $\mathbb F$ restrict to maps $\overline\partial_{i+1}:
\overline F_{i+1}\longrightarrow\overline F_i.$
Applying the Snake Lemma to the following commutative diagram of split exact sequences 
\begin{equation*}
    \begin{tikzcd}
   0 \arrow{r}{} & \overline{{F}}_{i+1} \arrow{r}{} \arrow{d}{\overline{\partial_{i+1}}}  &  {F}_{i+3} \arrow{r}{\alpha_{i+1,M}} \arrow{d}{\partial_{i+3}} & {F}_{i+1} \arrow{r}{} \arrow{d}{\partial_{i+1}} & 0 \\  0 \arrow{r}{} & \overline{{F}}_{i} \arrow{r}{} &
        {F}_{i+2} \arrow{r}{\alpha_{i,M}} & {F}_{i} \arrow{r}{} & 0,
    \end{tikzcd}
\end{equation*} 
we obtain $K_i\cong
\operatorname{coker}(
\overline F_{i+1}
\xrightarrow{\overline\partial_{i+1}}
\overline F_i)$
for all $i\geq n_0$. Thus the left-hand column forms a graded free resolution of $K_{n_0}$. Since $\mathbb F$ is minimal, $\partial_{i+3}(F_{i+3}) \subseteq\m F_{i+2}.$ As $\overline F_i$ is a graded direct summand of $F_{i+2}$, we have $\m F_{i+2}\cap\overline F_i=\m\overline F_i.$ Therefore,
$\overline\partial_{i+1}(\overline F_{i+1})
\subseteq\m\overline F_i,$
and hence the above graded free resolution is minimal. Consequently, its successive syzygies satisfy $K_{i+1}\cong\Syz_1^A(K_i)$
for all $i\geq n_0$. This proves part $(c)$.
\end{proof}

Note that in the case of complexity one, for all $i\gg0$, $\cx_A(K_i)=0$ and hence $\projdim_A(K_i)<\infty$. Since $K_i$ is maximal Cohen-Macaulay, it is grade free. Therefore, $K_{i+1}\cong \Syz^A_1(K_i)=0$ for all $i\gg0.$  Thus, $M_{i+2}\cong M_i(-w)$ for all $i\gg0.$

\section{Linearity of regularity}
In this section, under the hypotheses of \ref{hypo}, we study the asymptotic behavior of the Castelnuovo-Mumford regularity of the free modules $F_i$ and the syzygies $M_i$. We first prove that $\reg(F_i)$ eventually agrees with a linear polynomial. We then obtain linear lower and upper bounds for $\reg(M_i)$, and, under an additional suitable hypothesis, show that $\reg(M_i)$ is eventually linear.
\begin{point}
    \normalfont
Let $A$ be a standard $\N$-graded Noetherian $*$-local ring and $M$ be a finitely generated  $\Z$-graded $A$-module. Let
\[
\mathbb F:\quad
\cdots\longrightarrow F_i
\xrightarrow{\partial_i}
F_{i-1}
\longrightarrow\cdots\longrightarrow
F_0\longrightarrow M\longrightarrow0
\]
be the minimal graded free resolution of $M$ over $A$, where $\displaystyle F_i=\bigoplus_{s\in\mathbb{Z}}A(-s)^{\beta_{i,s}}$. Define $\displaystyle \alpha_i(M)=\operatorname{max}\{ s\in \Z\, |\, \beta_{i,s}\neq 0\}$  and $\displaystyle \gamma_i(M)= \operatorname{min}\{ s\in \Z\,|\, \beta_{i,s}\neq 0 \}.$
Note that $\alpha_i(M)$ and $\gamma_i(M)$ are the largest and smallest shifts of the $i$th free module in the graded minimal free resolution of $M.$ We say that $M$ has a \emph{pure resolution} if $\alpha_i(M)=\gamma_i(M)$ for all $i\geq 0$. 
Since $\mathbb F$ is minimal, we have $$\displaystyle \Ext_A^i(M,k)\cong\Hom_A(F_i,k)\cong
\bigoplus_{s\in\mathbb{Z}}k(s)^{\beta_{i,s}}$$ and $$\Tor^A_i(M,k)\cong F_i\otimes_A k\cong \bigoplus_{s\in\mathbb{Z}}k(-s)^{\beta_{i,s}}
$$ as graded $k$-vector spaces, for all $i\geq0$.  It follows that
$\text{end}(\Tor^A_i(M,k))=\alpha_i(M)$
and $\text{end}(\Ext_A^i(M,k))=-\gamma_i(M).$
\end{point}

\begin{proposition}\label{shift}
 Assume the same hypothesis as in \ref{hypo}. Then, for each $l\in\{0,1\}$, there exist $a_l,a'_l\in \{w_1,\ldots ,w_c\}$ and  $e_l,e'_l\in \Z$ such that $\alpha_{2i+l}(M)=a_l.i+e_l$
and $\gamma_{2i+l}(M)=a'_l.i+e'_l$ for all $i\gg0$.
\end{proposition}
\begin{proof}
    Note that, for all $i\geq 0,$ $\Tor^A_i(M,k)$ and $\Ext_A^i(M,k)$ are finitely generated graded $A$-modules of dimension zero. It follows that, for all $i\geq 0$, $$\reg(\Tor^A_i(M,k))=\text{end}(\Tor^A_i(M,k))=\alpha_i(M)$$ and $$\reg(\Ext_A^i(M,k))=\text{end}(\Ext_A^i(M,k))=-\gamma_i(M).$$ By Theorem A of \cite{irregularity}, for all $l\in \{0,1\}$ there exist $a_l,a'_l\in \{w_1,\ldots ,w_c\}$ and $e_l,e'_l\in \Z $ such that $\reg(\Ext_A^{2i+l}(M,k))=-a'_l.i+e'_l$ and $\reg(\Tor^A_{2i+l}(M,k))=a_l.i+e_l$ for all $i\gg 0.$ This implies that $\gamma_{2i+l}(M)=a'_l.i-e'_l$ and $\alpha_{2i+l}(M)=a_l.i+e_l$ for all $i\gg 0.$ This proves the result.
\end{proof}

\begin{corollary}\label{shiftcoro}
   Assume the same hypothesis as in \ref{hypo}. Then, for each $l\in\{0,1\}$, there exist integers $a_l\in \{w_1,\ldots ,w_c\}$ and $e_l\in \Z$ such that $\reg(F_{2i+l})=a_l.i+e_l$ for all $i\gg0$.
\end{corollary}

\begin{proof}
    Note that $\reg(A(-s))=\reg(A)+s$. Let \[
\mathbb F:\quad
\cdots\longrightarrow F_i
\xrightarrow{\partial_i}
F_{i-1}
\longrightarrow\cdots\longrightarrow
F_0\longrightarrow M\longrightarrow0
\]
be the minimal graded free resolution of $M$ over $A$, where $\displaystyle F_i=\bigoplus_{s\in\mathbb{Z}}A(-s)^{\beta_{i,s}}$. By Lemma \ref{reg}, $\reg(F_i)=\max\limits_{s}\{\reg(A(-s))\}=\reg(A)+\alpha_i(M).$ The result now follows from Proposition \ref{shift}. 
\end{proof}
\begin{remark}
   Note that Proposition \ref{shift} and Corollary \ref{shiftcoro} do not require $Q$ to be a polynomial ring over a field $k$; it suffices that $Q$ is a standard $\N$-graded $*$-local ring with $\text{projdim}_Q(M)<\infty$ (see \cite[Theorem A]{irregularity}).
\end{remark}

In the following results, we study the behavior of $\reg(M_i)$ for all $i\gg0$ under the hypothesis of \ref{hypo}.

\begin{theorem}\label{linearbound}
Assume the same hypothesis as in \ref{hypo}. Let $M$ be a graded maximal Cohen-Macaulay $A$-module. Then, for each $l\in\{0,1\}$, the function $i\longmapsto \reg(M_{2i+l})$ is eventually bounded above and below by linear polynomials.
\end{theorem}
\begin{proof}
   Since $A=Q/(\mathbf f)$ and $M_i$ is a finitely generated graded $A$-module, by
the graded independence theorem (see \cite[14.1.7]{Brodmann}), we have
$\reg(M_i)$ as an $A$-module is the same as
$\reg(M_i)$ computed as a $Q$-module.

Moreover, in the minimal graded free resolution $\F$,
the shifts occurring in $F_i$ are precisely the degrees of the minimal
homogeneous generators of $M_i$. Hence $\alpha_i(M)$ is the largest degree
of a minimal homogeneous generator of $M_i$.
Therefore, by \ref{regpoly}, $\reg(M_i)\geq \alpha_i(M)$ for all $i\geq0$.
Since $A$ is Cohen-Macaulay of dimension $d$, and both $F_i$ and $M_i$ are graded maximal Cohen-Macaulay $A$-modules for all $i\geq0$, we have
$$
\reg(F_i)=
\operatorname{end}(H^d_{A_+}(F_i))+d \, \text{ and }\, \reg(M_i)=\operatorname{end}(H^d_{A_+}(M_i))+d.
$$
For each $i\geq0$, we also have a short exact sequence of graded $A$-modules
$$
0\longrightarrow M_{i+1}\longrightarrow F_i
\longrightarrow M_i\longrightarrow0.
$$
Since $M_{i+1}$, $F_i$, and $M_i$ are maximal Cohen-Macaulay, the corresponding long exact sequence of graded local cohomology yields a short
exact sequence
$$
0\longrightarrow H^d_{A_+}(M_{i+1})\longrightarrow H^d_{A_+}(F_i)\longrightarrow H^d_{A_+}(M_i)
\longrightarrow0.
$$
Consequently, $\reg(F_i)=\max\{\reg(M_i),\reg(M_{i+1})\}$ for all $i\geq0$. In particular,
$\reg(M_i)\leq\reg(F_i)
=\reg(A)+\alpha_i(M).$ Combining this with $\alpha_i(M)\leq\reg(M_i),$ we obtain
$\alpha_i(M) \leq \reg(M_i)
\leq \alpha_i(M)+\reg(A)$
for all $i\geq0$. The result now follows from Proposition~\ref{shift}.

\end{proof}

In the previous result, we obtained eventual linear upper and lower bounds for $\reg(M_{2i+l})$, for $l=0,1$. In the next result, we consider the situation in which these bounds coincide asymptotically, so that the function $i\longmapsto \reg(M_{2i+l})$
is eventually given by a linear polynomial for each $l\in\{0,1\}$.

\begin{theorem}\label{Reg-syzygy}
Assume the same hypothesis as in \ref{hypo}. Let $\deg(f_j)=w$ for all $1\leq j\leq c$. Let $M$ be a graded maximal Cohen-Macaulay $A$-module. Then the function $i\longmapsto \reg(M_{2i+l})$
is eventually given by a linear polynomial for each $l\in\{0,1\}$. Moreover, both linear polynomials have leading coefficient $w$.
\end{theorem}

\begin{proof}
Assume that $\cx_A(M)=1$. By the proof of Proposition
\ref{cxM}, after increasing $n_0$ if necessary, we have $M_{i+2}\cong M_i(-w)$ for all $i\geq n_0$. Choose an even integer $n\geq n_0$. Iterating the
above isomorphism, we obtain
$M_{2i+n}\cong M_n(-iw)
\text{ and }
M_{2i+n+1}\cong M_{n+1}(-iw)$
for all $i\geq1$. Consequently,
$\reg(M_{2i+n})=
\reg(M_n)+iw$ and
$\reg(M_{2i+n+1})=\reg(M_{n+1})+iw$ for all $i\geq1$. Hence, both the even and odd subsequences of $\reg(M_i)$ are eventually linear with leading coefficient
$w$. This proves the result.

  Now assume that $\cx_A(M)\geq2$. By Theorem
\ref{linearbound}, for all $i\gg0$, say $i\geq n_1$, we have
$\alpha_i(M)\leq \reg(M_i)
\leq \alpha_i(M)+\reg(A).$
Write $\reg(M_i)=\alpha_i(M)+\delta_i$ for all $i\geq n_1$. Then
$0\leq\delta_i\leq\reg(A)$
for all $i\geq n_1$. By Proposition \ref{cxM}, for all $i\geq n_0$, there is a short exact sequence of graded $A$-modules
$$
0\longrightarrow K_i\longrightarrow M_{i+2}
\longrightarrow M_i(-w)\longrightarrow0.
$$
Since $M_i$ is a graded maximal Cohen-Macaulay $A$-module for all
$i\geq0$, the module $K_i$ is also a graded maximal Cohen-Macaulay $A$-module for all
$i\geq n_0$. As in the proof of Theorem \ref{linearbound}, we obtain
$$
\reg(M_{i+2})
=
\max\{\reg(M_i(-w)),\reg(K_i)\}
$$
for all $i\geq n_0$. In particular, $\reg(M_i(-w))
\leq \reg(M_{i+2}).$
By Proposition \ref{shift}, after increasing $n_1$ if necessary, for each
$l\in\{0,1\}$ there exists $e_l\in\mathbb Z$ such that
$\alpha_{2i+l}(M)=wi+e_l$
for all $i\geq n_1$. Let $n=\max\{n_0,n_1\}$. Then, for all $i\geq n$,
\begin{equation*}
\begin{split}
\reg(M_{2i})+w
&=\reg(M_{2i}(-w))
\leq \reg(M_{2i+2})\\
&\implies
\alpha_{2i}(M)+\delta_{2i}+w
\leq
\alpha_{2i+2}(M)+\delta_{2i+2}\\
&\implies
wi+e_0+w+\delta_{2i}
\leq
w(i+1)+e_0+\delta_{2i+2}\\
&\implies
\delta_{2i}\leq\delta_{2i+2}.
\end{split}
\end{equation*}
Thus $\{\delta_{2i}\}$ is an eventually non-decreasing sequence of
non-negative integers bounded above by $\reg(A)$. Hence it
is eventually constant. Therefore $i\longmapsto\reg(M_{2i})$ is eventually linear with leading coefficient $w$. The same argument applied to the odd subsequence shows that $\{\delta_{2i+1}\}$ is eventually constant. Hence $i\longmapsto\reg(M_{2i+1})$ is also eventually linear with leading coefficient $w$. This completes the proof. 
    
\end{proof}

\section{Linearity of \texorpdfstring{$V$}{v}-number}
In this section, under the hypotheses of \ref{hypo}, we first prove that, for each $l\in{0,1}$, the local and global $v$-numbers of $F_{2i+l}$ are eventually given by linear polynomials in $i$. We then show that the sets $\Ass_A(M_i)$ stabilize for all $i\gg0$. Finally, when $\deg(f_j)=w$ for all $1\leq j\leq c$, we prove that the local and global $v$-numbers of $M_{2i+l}$ are eventually linear with leading coefficient $w$.
\begin{definition}
\normalfont
Let $R$ be a standard graded algebra over a field $k$, and let $M$ be a finitely generated graded $R$-module. For $\mathfrak{p} \in \operatorname{Ass}_R(M)$,
the \textit{local Vasconcelos invariant} (or the local $v$-number) of $M$ at $\mathfrak{p}$ is defined as $v_{\mathfrak{p}}(M)
:= \inf\{n : \mathfrak{p} = (0 :_R x)
\text{ for some } x \in M_n\}.$
The \textit{Vasconcelos invariant} (or $v$-number) of $M$ is defined by
$$
v(M)=\min\left\{v_{\p}(M):\p\in\Ass_R(M)
\right\}.$$
By convention, $v(0)=\infty$ and $v_{\p}(M)=\infty$ if $\p\notin \Ass_A(M)$. 
From the definition of $v$-number, it can be easily checked that $v(M(-i))=v(M)+i$ and $v_{\p}(M(-i))=v_{\p}(M)+i$ for all $\p\in \Ass_R(M)$.
\end{definition}


\begin{lemma}
\label{v-ses}
Let $A$ be a graded Cohen-Macaulay ring, and let
$0\longrightarrow L\longrightarrow X\longrightarrow N\longrightarrow0$
be a short exact sequence of finitely generated graded maximal
Cohen-Macaulay $A$-modules. Then for each $\p\in \Ass_A(X)$,
$\min\{v_{\p}(L),v_{\p}(N)\}\leq v_{\p}(X)\leq v_{\p}(L).$ Moreover, $\min\{v(L),v(N)\}\leq v(X)\leq v(L).$
\end{lemma}

\begin{proof}
Let $\p\in \Ass_A(X).$
The inequality $v_{\p}(X)\leq v_{\p}(L)$ follows from Proposition~3.1(1) of \cite{ghoshv}. Note that if $\p\notin \Ass_A(L)$, then $v_{\p}(L)=\infty$ and the inequality is automatic.
We prove the other inequality. Choose a homogeneous element $x\in X$
such that $(0:_A x)=\p$ and $\deg(x)=v_{\p}(X).$ Since $X$ is maximal Cohen-Macaulay and $A$ is Cohen-Macaulay,
$\Ass_A(X)\subseteq\operatorname{Min}(A).$

If $x\in L$, then $v_{\p}(X)=\deg(x)\geq v_{\p}(L)$. It follows that $v_{\p}(X)=v_{\p}(L)$ and hence $v_{\p}(X)\geq\min\{v_{\p}(L),v_{\p}(N)\}.$ Suppose now that $x\notin L$, and let $\overline{x}$ denote its nonzero image in $N$. Since $(0:_A x)=\p$, we have $\p\subseteq(0:_A\overline{x}).$ Choose
$\q\in\Ass_A(A\overline{x}).$
As $A\overline{x}$ is a submodule of $N$, we have
$\q\in\Ass_A(N)
\subseteq\operatorname{Min}(A).$ Moreover, $\p\subseteq(0:_A\overline{x})\subseteq\q.$
Since both $\p$ and $\q$ are minimal primes of $A$, it follows that
$\p=(0:_A\overline{x})=\q.
$ Therefore
$\deg(x)=\deg(\overline{x})\geq v_{\p}(N).
$ Thus, in either case, $v_{\p}(X)\geq\min\{v_{\p}(L),v_{\p}(N)\}.$ This proves the local case. The global inequalities follow from Proposition 3.1(1) of \cite{ghoshv}, and the local inequalities proved above.
\end{proof}
By Lemma $\ref{v-ses}$, we can conclude that $v_{\p}(L\oplus N)=\inf\{v_{\p}(L),v_{\p}(N)\}$ for all $\p\in \Ass_R(M\oplus N)$ and $v(L\oplus N)=\inf\{v(M),v(N)\}$.

\begin{proposition}
\label{vFi}
Assume the same hypothesis as in \ref{hypo}. Then, for each $l\in\{0,1\}$, the function $i\longmapsto v_{\p}(F_{2i+l})$ is eventually given by a linear polynomial for all $\p\in \Ass(A)$. In particular, the function $i\longmapsto v(F_{2i+l})$ is eventually a linear polynomial.
\end{proposition}

\begin{proof}
Note that 
$\displaystyle F_i=\bigoplus_{s\in\mathbb Z}A(-s)^{\beta_{i,s}},$
and $\gamma_i(M)=
\min\{s\in\mathbb Z:\beta_{i,s}\neq0\}.$ Thus, $\Ass_A(F_i)=\Ass(A)$ for all $i\geq 0.$  Since $v(A(-s))=v(A)+s$, we have $v_{\p}(F_i)=v_{\p}(A)+\gamma_i(M)$ for all $\p\in \Ass(A)$ and $i\geq 0$. It follows that for all $i\geq 0$, $v(F_i)=v(A)+\gamma_i(M)$.
By Proposition \ref{shift}, for each $l\in\{0,1\}$, there exists $a_l\in \{w_1,\ldots ,w_c\}$ and
$e_l\in\mathbb Z$ such that
$\gamma_{2i+l}(M)=a_li+e_l$
for all $i\gg0$. Thus, for each $\p\in \Ass(A)$,
$v_{\p}(F_{2i+l})=v_{\p}(A)+\gamma_{2i+l}(M)
=a_li+(v_{\p}(A)+e_l)$
for all $i\gg0$. Consequently, $v(F_{2i+l})=v(A)+\gamma_{2i+l}(M)
=a_li+(v(A)+e_l)$
for all $i\gg0$. This proves the result.
\end{proof}

We begin by studying the asymptotic behavior of the associated primes of
the syzygies of a module over a complete intersection. The following
result shows that, for each $l\in \{0,1\}$, the sets $\Ass_A(M_{2i+l})$ eventually stabilize.

\begin{theorem} \label{stable-ass}
Let $(A,\m)=(Q/(\mathbf f),\n/(\mathbf f))$, where $(Q,\n)$ is a regular local ring with infinite residue field and
$\mathbf f=f_1,\ldots,f_c$ is a $Q$-regular sequence. Let $M$ be a finitely generated maximal Cohen-Macaulay $A$-module with $\cx_A(M)\geq1$, and set $M_i=\Syz_i^A(M)$ for $i\geq0$. Then, for each $l\in\{0,1\}$, the sets $\Ass_A(M_{2i+l})$ are eventually stable. Moreover, the result holds under hypothesis \ref{hypo}.
\end{theorem}

\begin{proof}
We prove by induction on $\cx_A(M)$. Suppose $\cx_A(M)=1$. Then $M_{i+2}\cong M_i$ for all $i\gg0$. Thus $\Ass_A(M_{i+2})=\Ass_A(M_i)$ for all $i\gg0$, and the result follows when $\cx_A(M)=1$.

Now, suppose $\cx_A(M)=r\geq2$ and the result holds for $r-1$. By \cite[Proposition 2.13]{asymptotic},  for all $i\gg0$ (say $i\geq n_0$), there is a short exact sequence
$$
0\longrightarrow K_i\longrightarrow M_{i+2}
\longrightarrow M_i\longrightarrow0,
$$
where $\cx_A(K_i)=r-1$ and $K_{i+1}\cong\Syz_1^A(K_i).$
Moreover, by the Depth Lemma, each $K_i$ is maximal Cohen-Macaulay. Since $K_{i+1}\cong \Syz_1^A(K_i)$, the induction hypothesis yields an integer $ n_1\geq n_0 $ such that, for each $l\in\{0,1\}$, the sets $\Ass_A(K_{2i+l})$ are constant for all $i\geq n_1$. Let $\mathcal B_l$ denote the stable value of $\Ass_A(K_{2i+l})$. For all $i\geq n_1$, the
above short exact sequence gives $\displaystyle \mathcal B_0 \subseteq \Ass_A(M_{2i+2})\subseteq \Ass_A(M_{2i})\cup \mathcal B_0.$

Consequently,
$\Ass_A(M_{2i+2})
\setminus\mathcal B_0 \subseteq \Ass_A(M_{2i}) \setminus \mathcal B_0.$
Thus, the sets
$\Ass_A(M_{2i})
\setminus\mathcal B_0$
form an eventually descending sequence of subsets of a finite set $\text{Min}(A)$, and hence they stabilize. Therefore,
$\Ass_A(M_{2i})$ is eventually stable. Similarly, we can show that the set $\Ass_A(M_{2i+1})$ is eventually stable by the same argument applied to the odd indices.

For the graded case, let $\mathfrak m=A_+$. Localizing hypothesis \ref{hypo} at $\mathfrak m$, we obtain a local complete intersection
$A_{\mathfrak m}$. Moreover,
$(M_i)_{\mathfrak m}\cong\Syz_i^{A_{\mathfrak m}}(M_{\mathfrak m}),$
and the complexity is preserved. Hence, by the local case, the sets
$\Ass_{A_{\mathfrak m}}\bigl((M_{2i+l})_{\mathfrak m}\bigr)$
are eventually stable for each $l\in\{0,1\}$. Since every associated prime of the graded module $M_i$ is homogeneous and therefore contained in $\mathfrak m$, we have
$\Ass_{A_{\mathfrak m}}((M_i)_{\mathfrak m})=\{\mathfrak pA_{\mathfrak m}\, |\,
\mathfrak p\in\Ass_A(M_i)\}.$
Taking contractions to $A$, the result follows.
\end{proof}

We now turn to the asymptotic behavior of the local and global $v$-numbers of syzygies. We need the following lemma to prove the main result. 
\begin{lemma}
\label{uniform-annihilator}
Assume the same hypotheses as in \ref{hypo} and let
$\varnothing \neq \mathcal B\subseteq \Ass(A)=\operatorname{Min}(A)$. Let
$\p\in\operatorname{Min}(A)\setminus\mathcal B$.
Then there exists a homogeneous element
$a\in A\setminus\p$ and an integer $N\geq1$ such that $a^N L=0$ for all finitely generated graded $A$-modules $L$ satisfying
$\Ass_A(L)\subseteq\mathcal B.$
\end{lemma}

\begin{proof}
For all $\q\in\mathcal B$, choose a homogeneous element
$a_{\q}\in\q\setminus\p,$ and set
$a=\prod_{\q\in\mathcal B}a_{\q}.$
Since $\p$ is prime, $a\notin\p$, while
$a\in\q$ for all $\q\in\mathcal B$. For each $\q\in\mathcal B$, $A_{\q}$ is a zero-dimensional local ring. Hence, its maximal ideal $\q A_{\q}$ is nilpotent. Thus, there exists an integer
$N\geq1$, independent of $L$, such that
$a^N=0 \text{ in }A_{\q}$
for all $\q\in\mathcal B$. Now, let $L$ be a finitely generated $A$-module with
$\Ass_A(L)\subseteq\mathcal B$. For all
$\q\in\Ass_A(L)$, we have $(a^NL)_{\q}=a^NL_{\q}=0.$
Consider the natural map $\displaystyle
L\longrightarrow
\bigoplus_{\q\in\Ass_A(L)}L_{\q}.$ Let $0\neq x\in L$ and $\q\in \Ass_A(Ax)\subseteq \Ass_A(L).$ It follows that $x/1\neq0$ in $L_{\q}$. Hence, the given map is injective. Since
$(a^NL)_{\q}=0$ for all $\q\in\Ass_A(L)$, it follows that $a^NL=0.$
\end{proof}

\begin{theorem}
\label{local-v-linearity}
Assume the same hypotheses as in \ref{hypo}, and suppose that $\deg(f_j)=w$ for all $1\leq j\leq c$. Let $M$ be a graded maximal Cohen-Macaulay $A$-module with
$\cx_A(M)\geq1$. For each $l\in\{0,1\}$, set
$\mathcal A_l=\Ass_A(M_{2i+l})
\text{ for }i\gg0$ (see Theorem \ref{stable-ass}). Then, for each $\p\in\mathcal A_l$, there exists $b_{\p,l}\in\mathbb Z$ such that
$v_{\p}(M_{2i+l})=wi+b_{\p,l}$ for all $i\gg0$. In particular, $\displaystyle v(M_{2i+l})= \min_{\p\in\mathcal A_l}\{v_{\p}(M_{2i+l}\}=wi+b_l$ for all $i\gg0$, where $\displaystyle b_l=\min_{\p\in\mathcal A_l}b_{\p,l}$.
\end{theorem}

\begin{proof}
We proceed by induction on
$r=\cx_A(M).$ Suppose that $r=1$. By Proposition \ref{cxM}, we have
$M_{i+2}\cong M_i(-w)$
for all $i\gg0$. Hence, for all
$\p\in\Ass_A(M_i)$,
$v_{\p}(M_{i+2})=v_{\p}(M_i)+w.$
Therefore, for each $l\in\{0,1\}$ and
$\p\in\mathcal A_l$, the function
$v_{\p}(M_{2i+l})$ is eventually linear with leading
coefficient $w$.

Now, suppose $\cx_A(M)=r\geq2$ and the result holds for $r-1$. By Proposition \ref{cxM}, for all $i\gg0$ (say $i\geq n_0$) there are
short exact sequences
\begin{equation}
\label{local-v-seq}
0\longrightarrow K_i
\longrightarrow M_{i+2}
\longrightarrow M_i(-w)
\longrightarrow0,
\end{equation}
where $\cx_A(K_i)=r-1$
and
$K_{i+1}\cong\Syz_1^A(K_i).$
Moreover, all the modules $M_i$ and $K_i$ are maximal Cohen-Macaulay. For each $l\in\{0,1\}$, let
$\mathcal B_l=\Ass_A(K_{2i+l})
\text{ for }i\gg0.$
Fix $\p\in\mathcal A_l.$

\medskip

\noindent
\textbf{Case 1.}
$\p\in\mathcal B_l$.

By the induction hypothesis,
there exists $c_{\mathfrak p,l}\in\mathbb Z$ such that
$v_{\mathfrak p}(K_{2i+l})=wi+c_{\mathfrak p,l}$
for all $i\gg0$. Hence
$v_{\mathfrak p}(K_{2i+l})=v_{\mathfrak p}(K_{2i+l-2})+w$ for all $i\gg0$. Applying \cite[Proposition~3.1(1)]{ghoshv} to the exact sequence
$$
0\longrightarrow K_{2i+l-2}
\longrightarrow M_{2i+l}
\longrightarrow M_{2i+l-2}(-w)
\longrightarrow0,
$$
we obtain $v_{\mathfrak p}(M_{2i+l})
\leq v_{\mathfrak p}(K_{2i+l-2}).$
Therefore,
$v_{\mathfrak p}(M_{2i+l})+w \leq
v_{\mathfrak p}(K_{2i+l}).
$ Now Lemma \ref{v-ses}, applied to
$0\longrightarrow K_{2i+l}
\longrightarrow M_{2i+l+2}
\longrightarrow M_{2i+l}(-w)
\longrightarrow0,$
gives
$v_{\mathfrak p}(M_{2i+l})+w
\leq v_{\mathfrak p}(M_{2i+l+2})
\leq v_{\mathfrak p}(K_{2i+l})$
for all $i\gg0$. Set
$\delta_i=v_{\mathfrak p}(M_{2i+l})-wi.$ The preceding inequality shows that $\{\delta_i\}$ is eventually
non-decreasing and bounded above by $c_{\mathfrak p,l}-w.$ Therefore, $\delta_i$ is eventually constant.
Consequently, there exists $b_{\mathfrak p,l}\in\mathbb Z$ such that
$v_{\mathfrak p}(M_{2i+l})=wi+b_{\mathfrak p,l}$
for all $i\gg0$.

\noindent
\textbf{Case 2.}
$\p\notin\mathcal B_l$.

By Proposition 3.1(2) of \cite{ghoshv} and \eqref{local-v-seq}, we have 
\begin{equation}\label{we}
     v_{\p}(M_{2i+l})\geq v_{\p}(M_{2i+l-2})+w \text{ for all } i\gg0.
\end{equation}
Now, Choose an integer $i_0$ of parity $l$ sufficiently large so that all the above stability
properties hold. For $t\geq1$, composing the surjections in
\eqref{local-v-seq}, we obtain
$M_{i_0+2t}
\twoheadrightarrow
M_{i_0}(-tw).$ Let $L_t$ denote its kernel.

\noindent
\textbf{Claim:} $\Ass_A(L_t)\subseteq\mathcal B_l.$

We proceed by induction on $t$. For $t=1$, we have $L_1=K_{i_0},$ and hence
$\Ass_A(L_1)=\mathcal B_l.$ Now, let $t\geq2$. Consider the surjection
$M_{i_0+2t}
\longrightarrow
M_{i_0+2t-2}(-w),$
whose kernel is $K_{i_0+2t-2}$. By the definition of $L_t$, the
preimage of $L_{t-1}(-w)$ under this surjection is precisely $L_t$.
Therefore, there is a short exact sequence
$$
0\longrightarrow K_{i_0+2t-2}
\longrightarrow L_t
\longrightarrow L_{t-1}(-w)
\longrightarrow0.
$$
Hence, by the induction hypothesis, the claim follows.

Since $\p\notin\mathcal B_l$, Lemma
\ref{uniform-annihilator} yields a homogeneous element
$a\in A\setminus\p$ and an integer $N\geq1$, independent of
$t$, such that $a^N L_t=0$ for all $t\geq1$. Choose a homogeneous element
$y\in M_{i_0}$ such that
$(0:_A y)=\p$  and $\deg(y)=v_{\p}(M_{i_0}).$
Regarded as an element of $M_{i_0}(-tw)$, its degree is
$v_{\p}(M_{i_0})+tw.$

Let $x\in M_{i_0+2t}$ be a homogeneous lift of $y$. For all $b\in\p$, we have
$bx\in L_t.$ Since $a^NL_t=0$, it follows that $b(a^Nx)=0.$ Hence, $\p(a^Nx)=0.$ Moreover, $a\notin\p=(0:_Ay)$, and hence $a^Ny\neq0.$ Thus,  $a^Nx\neq0$. Since $\Ass_A(M_{i_0+2r})\subseteq \text{Min}(A)$ and $\p\subseteq (0:_A a^Nx)\subseteq \q$, where $q\in \Ass_A(Aa^Nx)\subseteq \Ass_A(M_{i_0+2r})$, we obtain $(0:_A a^Nx)=\p.$
Therefore,
$$v_{\p}(M_{i_0+2t})
\leq \deg(a^Nx)=v_{\p}(M_{i_0})+tw+N\deg(a).$$
Hence, $ v_{\p}(M_{i_0+2t})-tw\leq v_{\p}(M_{i_0})+N\deg(a).$ Set
$\varepsilon_t=v_{\p}(M_{i_0+2t})-tw.$ Thus,   $\{\varepsilon_t\}$ is bounded above and, by \eqref{we}, non decreasing. Since
$\varepsilon_t\in\mathbb Z$, it is eventually constant. Thus there
exists $C\in\mathbb Z$ such that
$v_{\p}(M_{i_0+2t})=tw+C$
for all $t\gg0$. Writing $i_0=2j_0+l$ and replacing $t$ by $i-j_0$, we obtain
$v_{\p}(M_{2i+l})=wi+b_{\p,l}$
for some $b_{\p,l}\in\mathbb Z$ and all $i\gg0$. This completes the induction and the proof.
\end{proof}





\section{Asymptotic relations among regularity and \texorpdfstring{$v$}{v}-numbers}
In this section, we study the asymptotic relations between regularity and $v$-numbers. We compare the eventual leading coefficients of $v_{\mathfrak p}(F_{2i+l})$ and $\operatorname{reg}(F_{2i+l})$, and
show that, in the equal-degree case, the differences
$\reg(F_{2i+l})-v(F_{2i+l}) \text{ and } \reg(M_{2i+l})-v(M_{2i+l})$ are eventually constant. We conclude with a comparison in the Artinian case, where equality for the free modules characterizes pure resolutions. Finally, we present examples illustrating the
eventual linearity of regularity, local $v$-numbers, and the global
$v$-number.
\begin{proposition}
\label{v-reg-F}
Assume the same hypothesis as in \ref{hypo}. Then, for each $l\in\{0,1\}$, the eventual leading coefficient of $i\longmapsto v_{\p}(F_{2i+l})$ is at most the eventual leading coefficient of $i\longmapsto \reg(F_{2i+l})$, where $\p \in \Ass(A).$
\end{proposition}

\begin{proof}
Write
$F_i=\bigoplus_{s\in\mathbb Z}A(-s)^{\beta_{i,s}}.$
Then $v_{\p}(F_i)=v_{\p}(A)+\gamma_i(M)$ and $\reg(F_i)=\reg(A)+\alpha_i(M).$ By Proposition \ref{shift}, for each $l\in\{0,1\}$, there exist
$a_l,b_l\in\{w_1,\ldots,w_c\}$ and $e_l,e_l'\in\mathbb Z$ such that
$$
\alpha_{2i+l}(M)=a_li+e_l \text{ and } \gamma_{2i+l}(M)=b_li+e_l'
$$ for all $i\gg0$. Since $\gamma_{2i+l}(M)\leq\alpha_{2i+l}(M)$
for all $i$, it follows that $b_l\leq a_l.$
Therefore, the eventual leading coefficient of $v_{\p}(F_{2i+l})$ is at most the eventual leading coefficient of
$\reg(F_{2i+l})$. In particular, the eventual leading coefficient of $v(F_{2i+l})$ is at most the eventual leading coefficient of $\reg(F_{2i+l})$ for each $l\in \{0,1\}$.
\end{proof}
The preceding result compares the leading coefficients. In the equal-degree case, the corresponding differences are eventually constant.
\begin{proposition}
\label{difference-reg-v}
Assume the same hypotheses as in \ref{hypo}, and suppose that
$\deg(f_j)=w$ for all $1\leq j\leq c$. Let $M$ be a graded maximal
Cohen-Macaulay $A$-module.
Then, for each $l\in\{0,1\}$, both
$\reg(F_{2i+l})-v(F_{2i+l})$
and
$\reg(M_{2i+l})-v(M_{2i+l})$
are eventually constant.
\end{proposition}

\begin{proof}
For the free modules, we have
$\reg(F_i)=\reg(A)+\alpha_i(M)
$ and $v(F_i)=v(A)+\gamma_i(M).$
Hence, by Proposition \ref{shift}, for $i\gg0$,
\begin{align*}
\reg(F_{2i+l})-v(F_{2i+l})
=
\reg(A)-v(A)
+\alpha_{2i+l}(M)-\gamma_{2i+l}(M)
=
\reg(A)-v(A)+e_l-e_l'.
\end{align*}

On the other hand, by Theorem \ref{Reg-syzygy} and Theorem \ref{local-v-linearity}, we have 
$\reg(M_{2i+l})=wi+a_l$ and $v(M_{2i+l})=wi+b_l$, where $a_l,b_l\in \Z$,
for all $i\gg0$. Therefore,
$\reg(M_{2i+l})-v(M_{2i+l})
=a_l-b_l$ for all $i\gg0$.
\end{proof}

We next specialize to the Artinian case. In this setting, the $v$-number and the regularity of the ring coincide. This allows us to obtain a direct comparison between the regularity and the $v$-number of the free modules and the syzygies. Moreover, equality for the free
modules detects purity of the minimal graded free resolution.
\begin{proposition}
\label{reg-v-Fi-artinian}
Assume the same hypotheses as in \ref{hypo}, and suppose that $A$ is
Artinian. Then
\begin{enumerate}
\item $\reg(F_i)\geq v(F_i)$ for all $i\geq 0$. Moreover, equality holds for all $i\geq 0$ if and only if $M$ has a pure minimal graded free resolution.
\item $\reg(M_{2i+l})\geq v(M_{2i+l})$ for all $i\geq 0$.
\end{enumerate}
In particular, if $\deg(f_i)=w$ for all $1\leq i\leq c$, then for each $l\in\{0,1\}$, both
$\reg(F_{2i+l})-v(F_{2i+l})$
and
$\reg(M_{2i+l})-v(M_{2i+l})$
are eventually non-negative constants.
\end{proposition}

\begin{proof}
Let $N$ be a nonzero finitely generated graded $A$-module.  Since $A$ is Artinian, so is $N$. It follows that $\Ass_A(N)=\{\m\}$. Hence $v(N)=v_{\m}(N)$. Let $\m=(0:_{N}x)$ such that $v(N)=\deg(x)$. Since $N$ is Artinian, it follows that $\deg(x)\leq \text{end}(N)=\reg(N).$ This implies that $\reg(F_i)\geq v(F_i)$ and $\reg(M_{2i+l})\geq v(M_{2i+l})$ for all $i\geq 0$.

Since $A$ is an Artinian complete intersection, hence Gorenstein. By Theorem 3.2 of \cite{saha}, 
$v(A)=\reg(A).$ Write
$F_i=\bigoplus_{s\in\mathbb Z}A(-s)^{\beta_{i,s}}.$
Then $\reg(F_i)=\reg(A)+\alpha_i(M)$
and
$v(F_i)=v(A)+\gamma_i(M).$
Hence,
$\reg(F_i)-v(F_i)=\alpha_i(M)-\gamma_i(M)$ for all $i\geq 0.$ Furthermore,
equality holds if and only if $\alpha_i(M)=\gamma_i(M)$ for all $i\geq 0.$ This proves the result.
\end{proof}

We now give two examples illustrating the linearity of regularity, local $v$-numbers, and the global $v$-number in our setting. The first example concerns the case $\cx_A(M)=1$, while the second treats the case $\cx_A(M)=2$. In the second example, we use the total complex; see \cite[1.2.6]{weibel} for its definition.

\begin{example}
\normalfont
Let $A=k[X,Y]/(X^2Y),$ where $k$ is an infinite field, and let $x,y$ denote the images of $X,Y$ in $A$. Note that here $w=3.$ Set $M=A/(x)\oplus A/(y).$ Since $A/(x)\cong k[y]\text{ and }A/(y)\cong k[x],$ both summands are one-dimensional Cohen-Macaulay $A$-modules. Hence $M$ is a graded maximal Cohen-Macaulay $A$-module. The minimal graded free resolutions of $A/(x)$ and $A/(y)$ are induced by the matrix factorizations
$X^2Y=X\cdot XY \text{ and } X^2Y=Y\cdot X^2,$
respectively; see \cite[Theorem~6.1]{Eisenbud}. Let $\mathbb F$ and $\mathbb G$ denote these resolutions. Then $\mathbb F\oplus\mathbb G$ is the minimal graded free resolution of $M.$ In particular, $\beta_i^A(M)=2$ for every $i\geq0$, and hence $\cx_A(M)=1.$

For $i\geq0$, set $M_i=\operatorname{Syz}_i^A(M),\,\, M_0=M.$ The two matrix factorizations give
$M_{2i}\cong\bigl(A/(x)\oplus A/(y)\bigr)(-3i)$ and
$M_{2i+1}\cong \bigl(A/(xy)\oplus A/(x^2)\bigr)(-(3i+1))$
for every $i\geq0$. Consequently,
$\operatorname{Ass}_A(M_i)=\{(x),(y)\}$ for every $i\geq0$. We have $v_{(x)}(M)=v_{(y)}(M)=0.$
Moreover, $v_{(x)}\bigl(A/(xy)\oplus A/(x^2)\bigr)=
v_{(y)}\bigl(A/(xy)\oplus A/(x^2)\bigr)=1.$
Therefore,
$v_{(x)}(M_{2i})=v_{(y)}(M_{2i})=3i$ and $v_{(x)}(M_{2i+1})=v_{(y)}(M_{2i+1})=3i+2.$ It follows that
$v(M_{2i})=3i\text{ and }v(M_{2i+1})=3i+2.$

Furthermore,
$\reg(A/(x)\oplus A/(y))=0$ and $\reg(A/(xy)\oplus A/(x^2))=1.$ Thus, $\reg(M_{2i})=3i$ and $\reg(M_{2i+1})=3i+2.$ Hence, for every $\mathfrak p\in\{(x),(y)\}$, $l\in\{0,1\}$, and $i\geq0$, we have
$v_{\mathfrak p}(M_{2i+l})=v(M_{2i+l})=\reg(M_{2i+l}).$

Since $(\F\oplus \G)_{2i+l}=A(-(3i+l))^2$ for each $l\in\{0,1\}$ and
$\reg(A)=v_{\p}(A)=v(A)=2$ for every $\p\in\Ass(A)$, we have
$\reg(F_{2i+l})=3i+l+2,$ and similarly,
$v_{\p}(F_{2i+l})=v(F_{2i+l})=3i+l+2.$
\end{example}

\begin{example}\label{exam2}
\normalfont
Let $A=k[X,Y,Z]/(X^2,YZ),$ where $k$ is an infinite field, and let $x,y,z$ denote the images of $X,Y,Z$ in $A$. Let $M=A/(x,y).$ Note that here $w=2$ and $M$ is a graded maximal Cohen-Macaulay $A$-module. We also have $\Ass(A)=\{\p=(x,y),\q=(x,z)\}$ and $\Ass_A(M)=\{\p\}$.

Set
$R={k[X]}/{(X^2)}
\text{ and }S={k[Y,Z]}/{(YZ)}.$ By  \cite[Chapter~III, \S4, Proposition~1]{Bourbaki}, we have
$A\cong R\otimes_k S$ and $M\cong R/(x)\otimes_k S/(y).$ Now, we have 
$$\mathbb F:
\cdots
\xrightarrow{x}R(-2)
\xrightarrow{x}R(-1)
\xrightarrow{x}R
\longrightarrow R/(x)
\longrightarrow0.
$$ and 
$$\mathbb G:
\cdots
\xrightarrow{y}S(-2)
\xrightarrow{z}S(-1)
\xrightarrow{y}S
\longrightarrow S/(y)
\longrightarrow0,
$$
 are the minimal graded free resolutions of $R/(x)$ over $R$ and $S/(y)$ over $S$, respectively. Let $F_i$ and $G_i$ be the $i$th graded free module of $\F$ and $\mathbb G$, respectively. Note that $F_i=R(-i)$ and $G_i=S(-i)$ for all $i\geq 0.$ By \cite[Lemma~3.6]{Bergh} and the minimality of $\F$ and $\G$, the total complex
$\operatorname{Tot}(\F \otimes_k \mathbb G)$
is the minimal graded free resolution of
$R/(x)\otimes_kS/(y)\cong M$
over $R\otimes_kS\cong A.$ Note that for all $n\geq0$, the  $n$-th graded free module of $\operatorname{Tot}(\F \otimes_k \mathbb G)$ is $$T_n=\bigoplus_{a+b=n}F_a\otimes_kG_b=\bigoplus_{a+b=n}R(-a)\otimes_k S(-b)\cong
\bigoplus_{a+b=n}A(-(a+b))=A(-n)^{n+1}.$$ 
Therefore, $\beta_n^A(M)=n+1$ for all $n\geq0$, and hence $\cx_A(M)=2.$ For $i\geq1$, set
$M_i=\Syz_i^A(M).$
Since $M_i\subseteq T_{i-1},$ we have
$\Ass_A(M_i)
\subseteq \Ass_A(T_{i-1})=\Ass(A)=\{\p,\q\}.$
On the other hand,
$(0:_A xz)=(x,y)=\p$ and
$(0:_A xy)=(x,z)=\q.$

Let $e$ be the basis element corresponding to the summand $F_{i-1}\otimes_kG_0\subseteq T_{i-1}.$ For $i\geq2$, the differential on this summand is induced by multiplication by $x$. Therefore, 
$d(xz\,e)=0 \text{ and } d(xy\,e)=0.$
Thus, $xz\,e,\ xy\,e\in M_i,$
and their annihilators are $\p$ and $\q$,
respectively. The same conclusion is immediate for $i=1$. Hence
$\Ass_A(M_i)=\{\p,\q\}$
for all $i\geq1$.

We next compute the local $v$-numbers. Notice that
$(0:_A\p)=(xz)\text{ and }(0:_A\q)=(xy).$
Since
$T_{i-1}=A(-(i-1))^i,$
every nonzero homogeneous element of
$(0:_{T_{i-1}}\p)$ has degree at least
$(i-1)+2=i+1.$ As $M_i\subseteq T_{i-1}$, by \cite[Proposition 3.1]{ghoshv}, we have
$v_{\p}(M_i)\geq i+1.$
On the other hand, $xz\,e\in M_i$ has degree $i+1$ and annihilator
$\p$. This implies that
$v_{\p}(M_i)=i+1.$
Similarly, using $(0:_A\q)=(xy)$ and the element $xy\,e\in M_i$, we
obtain $v_{\q}(M_i)=i+1.$
Consequently, $v(M_i)=i+1$
for all $i\geq1$.

We next compute the regularity. Note that by \ref{hypo}, $\reg(A)=2$ and since $M\cong k[z]$, $\reg(M)=0.$ Since $M_i$ is a graded maximal Cohen-Macaulay module, from the exact sequence of graded $A$-modules $$0\longrightarrow M_{i+1}\longrightarrow T_i\longrightarrow M_i\longrightarrow 0,$$ we obtain $\reg(T_i)=\max\{\reg(M_{i+1}),\reg(M_i)\}$ for all $i\geq 0.$ Since $T_i\cong A(-i)^{i+1}$, it follows that $\reg(T_i)=i+\reg(A)$. Now,   $2=\reg(T_0)=\max\{\reg(M_{1}),0\}.$ It follows that $\reg(M_1)=2$. Inductively, we can conclude that $\reg(M_i)=i+1$ for all $i\geq 1$.

Thus
$\reg(M_i)=v(M_i)=v_{\p}(M_i)=v_{\q}(M_i)=i+1$
for all $i\geq1$. In particular,
$\reg(M_{2i})=v(M_{2i})=v_{\p}(M_{2i})
=v_{\q}(M_{2i})=2i+1=wi+1$
for $i\geq1$ and $\reg(M_{2i+1})=v(M_{2i+1})=v_{\p}(M_{2i+1})=v_{\q}(M_{2i+1})=2i+2=wi+2$
for $i\geq0$, where $w=2=\text{deg}(X^2)=\text{deg}(YZ)$.

Since $T_{2i+l}=A(-(2i+l))^{2i+l+1},$ for each $l\in \{0,1\}$ and $\reg(A)=v_{\p}(A)=v(A)=2$ for every $\p\in\Ass(A)$, we obtain $\reg(T_{2i+l})=v_{\p}(T_{2i+l})=v(T_{2i+l})=2i+l+2.$
\end{example}

\section*{Acknowledgements}
The author would like to thank generative AI for its assistance in the preparation of this paper. In particular, it was used to assist with the proof of Theorem \ref{local-v-linearity}, to suggest Lemma \ref{uniform-annihilator} and Example \ref{exam2}, and to improve the writing of the manuscript.

\providecommand{\bysame}{\leavevmode\hbox to3em{\hrulefill}\thinspace}
\providecommand{\MR}{\relax\ifhmode\unskip\space\fi MR }
\providecommand{\MRhref}[2]{
  \href{http://www.ams.org/mathscinet-getitem?mr=#1}{#2}
}

\end{document}